\documentclass[12pt]{amsart}
\usepackage{amsmath,amssymb,bm,amsthm,mathrsfs}
\newtheorem{theorem}{Theorem}[section]
\newtheorem{corollary}[theorem]{Corollary}
\newtheorem{remark}[theorem]{Remark}
\numberwithin{equation}{section}
\numberwithin{table}{section}
\allowdisplaybreaks[4]
\begin{document}
\def\d{\displaystyle}
\def\R{\mathbb{R}}
\def\bbK{\mathbb{K}}
\def\argmin{\mathop{\rm argmin}}
\def\fA{\mathscr{A}\hspace{-2pt}}
\def\fK{\mathfrak{K}}
\def\cK{\mathscr{K}\hspace{-4pt}}
\def\TG{T_{\scriptscriptstyle G}}
\def\EG{E_{\scriptscriptstyle G}}
\def\muG{\mu_{\scriptscriptstyle \bar{G}}}
\def\cT{\mathcal{T}}
\def\tbar{|\!|\!|}
\def\LT{{L_2(I)}}
\def\HO{{H^1(I)}}
\def\bq{{\bm q}}
\def\bp{{\bm p}}
\def\bw{{\bm w}}
\def\p{\partial}
\def\O{\Omega}
\def\bPhi{{\bm \Phi}}
\def\Div{\textrm{div}\hspace{1pt}}
\def\Curl{\textrm{curl}\hspace{1pt}}
\def\bG{\bm{G}}
\def\bK{\bm{K}}
\def\bn{{\bm n}}
\def\ker{\textrm{Ker}({\rm div})}
\title[A One Dimensional Optimal Control Problem with Derivative Constraints]
{A One Dimensional Elliptic Distributed Optimal Control Problem with
Pointwise Derivative Constraints}
\author{S.C. Brenner}
\address{Susanne C. Brenner, Department of Mathematics and Center for
Computation and Technology, Louisiana State University, Baton Rouge,
LA 70803, USA}
\email{brenner@math.lsu.edu}
\author{L.-Y. Sung}
\address{Li-yeng Sung,
 Department of Mathematics and Center for Computation and Technology,
 Louisiana State University, Baton Rouge, LA 70803, USA}
\email{sung@math.lsu.edu}
\author{W. Wollner}
\address{Winnifried Wollner, Department of Mathematics, Technische
Universit\"at Darmstadt, 64293 Darmstadt, Germany}
\email{wollner@mathematik.tu-darmstadt.de}
\thanks{The work of the first and second authors was supported in part
 by the National Science Foundation under Grant No.
 DMS-19-13035.}
\begin{abstract}
  We consider a one dimensional elliptic distributed optimal control
  problem with pointwise constraints on the
  derivative of the state.  By exploiting the variational inequality
  satisfied by the derivative
  of the optimal state, we obtain higher regularity for the optimal state
  under appropriate assumptions on the data.
  We also solve the optimal control problem as a fourth order variational
  inequality by a $C^1$ finite element method, and present the error analysis
  together with numerical results.
\end{abstract}
\maketitle
%
\section{Introduction}\label{sec:Introduction}
 Let $I$ be the interval $(-1,1)$ and the
 function $J:L_2(I)\times L_2(I)\longrightarrow \R$ be defined by
\begin{equation}\label{eq:JDef}
  J(y,u)=\frac12 \big(\|y-y_d\|_\LT^2+\beta\|u\|_\LT^2\big),
\end{equation}
 where  $y_d\in L_2(I)$ and $\beta$ is a positive constant.
\par
 The optimal control problem is to
\begin{equation}\label{eq:OCP}
  \text{find}\quad(\bar y,\bar u)=\argmin_{(y,u)\in \bbK}J(y,u),
\end{equation}
 where $(y,u)\in H^1_0(I)\times L_2(I)$ belongs to $\bbK$ if and only if
\begin{alignat}{3}
  \int_I y'z'dx&=\int_I (u+f)z\,dx&\qquad&\forall\,z\in H^1_0(I),
  \label{eq:PDEConstraint}\\
  y'&\leq \psi &\qquad&\text{a.e. on $I$}.\label{eq:DerivativeConstraint}
\end{alignat}
 We assume that
\begin{equation}\label{eq:DataRegularity}
  f\in H^1(I), \; \psi\in H^2(I)
\end{equation}
 and
\begin{equation}\label{eq:psiConstraint}
  \int_I \psi\,dx>0.
\end{equation}
\begin{remark}\label{rem:Literature}\rm
 The optimal control problem  defined by
 \eqref{eq:JDef}--\eqref{eq:DerivativeConstraint}
  is a  one dimensional analog
 of the optimal control problems considered in
  \cite{CB:1988:Gradient,CF:1993:Gradient,DGH:2009:Gradient,
  OW:2011:Gradient,Wollner:2012:Gradient}.
  It was solved by a $C^1$ finite element method in
  \cite{BSW:2020:OneD} under the assumptions that
\begin{equation}\label{eq:OldData}
  f\in H^{\frac12-\epsilon}(I)\quad\text{and}\quad
  \psi\in H^{\frac32-\epsilon}(I).
\end{equation}
\end{remark}
\par\medskip
 Since the constraint \eqref{eq:PDEConstraint} implies $y\in H^2(I)$
  by elliptic regularity,
 we can reformulate the optimization problem
 \eqref{eq:JDef}--\eqref{eq:DerivativeConstraint}
 as follows:
\begin{equation}\label{eq:ROCP}
  \text{Find}\quad \bar y=\argmin_{y\in K}\frac12\big(\|y-y_d\|_\LT^2
  +\beta\|y''+f\|_\LT^2\big),
\end{equation}
 where
\begin{equation}\label{eq:KDef}
  K=\{y\in H^2(I)\cap H^1_0(I):\,y'\leq \psi \;\text{on}\;I\}.
\end{equation}
\par
 According to the standard theory \cite{ET:1999:Convex,KS:1980:VarInequalities},
 the minimization problem defined by \eqref{eq:ROCP}-\eqref{eq:KDef}
  has a unique solution characterized by the
 fourth order variational inequality
\begin{equation*}
   \beta\int_I (\bar y''+f)(y''-\bar y'')dx+ \int_I (\bar y-y_d)
   (y-\bar y)dx\geq 0
    \qquad\forall\,y\in K,
\end{equation*}
 which can also be written as
\begin{equation}\label{eq:VI}
  a(\bar y,y-\bar y)\geq \int_I y_d(y-\bar y)dx-\beta\int_I f(y''-\bar y'')dx
  \qquad\forall\,y\in K,
\end{equation}
 where
\begin{equation}\label{eq:aDef}
  a(y,z)=\beta\int_I y'z'dx+\int_I yz\,dx.
\end{equation}
\begin{remark}\label{rem:VI}\rm
  The reformulation of state constraint optimal control problems as fourth order variational inequalities
  was discussed in \cite{PS:1996:OC}, and a nonconforming finite element based on this idea was introduced in
  \cite{LGY:2009:Control}.  Other finite element methods can be found in
  \cite{GY:2011:State,BSZ:2013:OptimalControl,BDS:2014:PUMOC,BOPPSS:2016:OC3D,BGPS:2018:Morley,
 BGS:2018:Nonconvex,BSZ:2019:Neumann}.
\end{remark}
\begin{remark}\label{rem:Nontrivial}\rm
Note that \eqref{eq:DerivativeConstraint} implies
\begin{equation*}
  \int_I\psi\,dx\geq \int_I y'\,dx=0 \qquad \forall\,y\in K
\end{equation*}
  and hence $\int_I \psi\,dx\geq0$ is a necessary condition for $K$
  to be nonempty.  It is also
   a sufficient condition because the function $y$ defined by
\begin{equation*}
  y(x)=\int_{-1}^x (\psi(t)-\bar\psi)dt
\end{equation*}
 belongs to $K$, where $\bar\psi$ is the mean of $\psi$ over $I$.
 Furthermore,
\begin{equation*}
  0=\int_I \psi\,dx=\int_I(\psi-y')dx
\end{equation*}
 together with \eqref{eq:DerivativeConstraint} implies $\psi=y'$
 identically on $I$ and hence
 $K=\{\psi\}$ is a singleton.   Therefore we impose the condition
 \eqref{eq:psiConstraint} to
 ensure that the optimization problem defined by
 \eqref{eq:ROCP}--\eqref{eq:KDef} is nontrivial.
\end{remark}
\par
 Our goal is to show that $\bar y\in H^3(I)$ under the assumptions in
 \eqref{eq:DataRegularity}
 and consequently \eqref{eq:ROCP}/\eqref{eq:VI} can be solved by a $C^1$
  finite element method
 with $O(h)$ convergence in the energy norm.  Note that previously
 $\bar y\in H^{\frac52-\epsilon}$ was the best regularity result in the
  literature for Dirichlet elliptic
 distributed optimal control
 problems on  smooth/convex domains with pointwise constraints on
 the gradient of the state.
\par
 The rest of the paper is organized as follows.
    The $H^3$ regularity of $\bar y$ is obtained in
 Section~\ref{sec:DerivativeVI} through a variational inequality for $\bar y'$
  that can be interpreted
 as a Neumann obstacle problem for the Laplace operator.  The $C^1$ finite
 element method for
 \eqref{eq:ROCP}/\eqref{eq:VI} is analyzed in Section~\ref{sec:Discrete},
  followed by numerical results
 in Section~\ref{sec:Numerics}.  We end with some remarks on the extension
  to higher dimensions
 in Section~\ref{sec:Conclusions}.
\section{A Variational Inequality for $\bar y'$}\label{sec:DerivativeVI}
 Observe that the set $\{y':\,y\in K\}$
  is the subset $\cK$ of $H^1(I)$
 given by
\begin{align}\label{eq:cKDef}
  \cK&=\{v\in H^1(I):\,\int_I v\,dx=0\quad\text{and}\quad v\leq\psi
  \;\text{on}\;I\},
\end{align}
 and the variational inequality \eqref{eq:VI} is equivalent to
\begin{align}\label{eq:NeumannVI}
  &\int_I (\Phi-f')(q-p)dx+\int_I p'(q'-p')dx\\
    &\hspace{40pt}+[f(1)(q(1)-p(1))-f(-1)(q(-1)-p(-1))]\geq0 \qquad
    \forall\,q\in \cK,\notag
\end{align}
 where $p=\bar y'$, $q=y'$, and $\Phi\in H^1(I)$ is determined by
\begin{align}
 \beta\Phi'&=y_d-\bar y \label{eq:PhiDef1}
\intertext{and}
 \int_I\Phi\,dx&=0.\label{eq:PhiDef2}
\end{align}
 Moreover \eqref{eq:NeumannVI} is the variational inequality that
 characterizes the solution of the following
 minimization problem:
\begin{equation}\label{eq:NeumannObstacle}
  \text{Find}\quad p=\argmin_{q\in \cK}\Big[\frac12 \int_I (q')^2dx
  +\int_I(\Phi-f')q\,dx
   +f(1) q(1)- f(-1) q(-1)\Big].
\end{equation}
\subsection{A Neumann Obstacle Problem}\label{subsec:Neumann}
 The minimization problem \eqref{eq:NeumannObstacle}, which is a
  Neumann obstacle problem, can be written more conveniently as
\begin{equation}\label{eq:ObstacleProblem}
  p=\argmin_{q\in \cK}\Big[\frac12 b(q,q)+(\phi,q)+ \tau q(1)
  -\sigma q(-1)\Big],
\end{equation}
 where $\sigma=f(-1)$, $\tau=f(1)$,
\begin{equation}\label{eq:Shorthands}
 b(q,r)=\int_I  q'r'\,dx,\quad
 (\phi,q)=\int_I \phi\, q\,dx \quad\text{and}\quad
 \phi=\Phi-f'.
\end{equation}
 Note that we have a compatibility condition
\begin{equation}\label{eq:Compatibility}
 \int_I \phi\,dx+\tau-\sigma=0
\end{equation}
 that follows from \eqref{eq:DataRegularity}, \eqref{eq:PhiDef2}
  and the Fundamental Theorem
  of Calculus for absolutely continuous functions.
\par
 Since $b(\cdot,\cdot)$ is coercive on $H^1(I)/\R$,
 the obstacle problem defined by \eqref{eq:cKDef}
 and \eqref{eq:ObstacleProblem}
 has a unique
 solution $p$ characterized by the variational inequality
\begin{equation}\label{eq:VI1}
 b(p,q-p)+(\phi,q-p)+\tau(q(1)-p(1))-\sigma(q(-1)-p(-1))\geq 0
 \qquad\forall\,q\in \cK.
\end{equation}
\begin{theorem}\label{thm:pRegularity}
  The solution $p=\bar y'\in \cK$ of \eqref{eq:ObstacleProblem}/\eqref{eq:VI1}
   belongs to $H^2(I)$.
\end{theorem}
\begin{proof}
 We begin by observing that
\begin{equation}\label{eq:VI2}
 b(p,q-p)+(\phi,q-p)+\tau(q(1)-p(1))-\sigma(q(-1)-p(-1))\geq 0
  \qquad\forall\,q\in\tilde K,
\end{equation}
 where
\begin{equation}\label{eq:tKDef}
  \tilde K=\{q\in H^1(I):\,q\leq\psi\;\;\text{in}\;I\;\;
  \text{and}\;\int_I q\,dx\geq0\}.
\end{equation}
 Indeed, $q\in \tilde K$ implies $q-\bar q\in K$, where $\bar q$
 is the mean of $q$ over $I$,
 and hence, in view of \eqref{eq:Compatibility}, the definition of
 $b(\cdot,\cdot)$ in
 \eqref{eq:Shorthands} and \eqref{eq:VI1},
\begin{align*}
    &b(p,q-p)+(\phi,q-p)+\tau(q(1)-p(1))-\sigma((q(-1)-p(-1))\\
      &\hspace{10pt}=b(p,q-\bar q-p)+(\phi,q-\bar q-p)+
      \tau(q(1)-\bar q-p(1))-\sigma(q(-1)-\bar q-p(-1))\\
      &\hspace{10pt}\geq 0
\end{align*}
 for all $q\in\tilde{K}$.
\par
 Let $\fK\subset H^1(I)$ be defined by
\begin{equation}\label{eq:bKDef}
 \fK=\{q\in H^1(I):\;q\leq \;\psi\;\;\text{in}\;I\},
\end{equation}
 and $G:H^1(I)\longrightarrow [0,\infty)$ be defined by
\begin{equation}\label{eq:gDef}
   G(q)=\int_I q\,dx.
\end{equation}
 Then the function $\psi$ belongs to $\fK$ and
\begin{equation}\label{eq:Slater}
  G(\psi)>0
\end{equation}
 by \eqref{eq:psiConstraint}.
\par
  It follows from the Slater condition \eqref{eq:Slater} and
  the theory of Lagrange multipliers
 \cite[Chapter 1, Theorem~1.6]{IK:2008:Lagrange} that there exists
 a nonnegative number $\lambda$
 such that
\begin{equation}\label{eq:VI3}
  b(p,q-p)+(\phi,q-p)+\tau((q(1)-p(1))-\sigma(q(-1)-p(-1))-\lambda\int_I (q-p)dx\geq 0
\end{equation}
 for all $q\in \fK$.
\par
 Finally we observe that
\begin{equation}\label{eq:VI4}
 \tilde b(p,q-p)+(F,q-p)+\tau(q(1)-p(1))-\sigma(q(-1)-p(-1))
 \geq 0\qquad\forall\,q\in\fK,
\end{equation}
 where
\begin{equation}\label{eq:taDef}
 \tilde b(q,r)=\int_I q'r'dx+\int_I qr\,dx
\end{equation}
 and
\begin{equation}\label{eq:FDef}
  F=\phi-\lambda-p.
\end{equation}
\par
 The variational inequality defined by \eqref{eq:bKDef},
 \eqref{eq:VI4} and \eqref{eq:taDef}
  characterizes the solution of
 a coercive Neumann  obstacle problem on $H^1(I)$.  Since
 $F\in \LT$ and $\psi\in H^2(I)$,
 we can apply the result in
 \cite[Chapter~5, Theorem~3.4]{Rodrigues:1987:Obstacle} to conclude that
 $p\in H^2(I)$.
\end{proof}
\par
 We can deduce the regularity of $(\bar y,\bar u)$ from the relations
 $p=\bar y'$ and $\bar u=-(\bar y''+f)$.
\begin{corollary}\label{cor:Regularity}
Under the assumption \eqref{eq:DataRegularity} on the data,
  the solution $(\bar y,\bar u)$ of the optimal control problem
  \eqref{eq:JDef}--\eqref{eq:psiConstraint}
  belongs to $H^3(I)\times H^1(I)$.
\end{corollary}
\begin{remark}\rm
  The result in \cite{Rodrigues:1987:Obstacle}, which is for
  dimensions $\geq 2$,
   requires a compatibility condition between $\p\psi/\p n$ and the
  Neumann boundary condition so that the boundary trace of
  the normal derivative of the
  solution of the obstacle problem
  belongs to the
 correct Sobolev space.
 This is not needed in one dimension since
  the boundary values of the normal derivative are just numbers.
\end{remark}
\subsection{The Karush-Kuhn-Tucker Conditions}\label{subsec:KKT}
 It follows from \eqref{eq:Shorthands}, Theorem~\ref{thm:pRegularity}
 and integration by parts that
\begin{equation}\label{eq:ObstacleLM}
  b(p,q)+(\phi,q)+\tau q(1)-\sigma q(-1)-\lambda\int_I q\,dx
  +\int_I q\,d\nu=0 \qquad\forall\,q\in H^1(I),
\end{equation}
 where the regular Borel measure $\nu$ is given by
\begin{equation}\label{eq:ObstaclenuDef}
  d\nu=(p''-\phi+\lambda)dx+[p'(-1)+\sigma]d\delta_{-1}-[p'(1)
  +\tau]d\delta_{1},
\end{equation}
 and $\delta_{-1}$ (resp., $\delta_1$) is the Dirac point measure at
 $-1$ (resp., $1$).
\par
 Let $\fA$ be the  active set of the derivative constraint
 \eqref{eq:DerivativeConstraint},
  i.e.,
\begin{equation}\label{eq:ActiveSet}
  \fA=\{x\in [-1,1]:\,\bar y'(x)=\psi(x)\}=\{x\in[-1,1]:\,p(x)=\psi(x)\}.
\end{equation}
 By a standard argument, $p$ satisfies \eqref{eq:VI3} if and only if
\begin{equation}\label{eq:ObstaclenuProperties}
  \text{$\nu$ is nonnegative and supported on $\fA$.}
\end{equation}
\par
 We can translate \eqref{eq:ObstacleLM}--\eqref{eq:ObstaclenuProperties}
  into  Karush-Kuhn-Tucker (KKT) conditions for the solution
  $\bar y'=p\in \cK$ of
 \eqref{eq:NeumannObstacle}/\eqref{eq:VI1}, which is summarized
 in the following theorem.
\begin{theorem}\label{thm:KKT}
  There exists a nonnegative number $\lambda$ such that
\begin{align}
 &\int_I p'q'dx+\int_I(\Phi-f')q\,dx+f(1)q(1)-f(-1)q(-1)+
 \int_I q\,d\nu\label{eq:KKT1}\\
 &\hspace{100pt}=\lambda\int_I q\,dx&\qquad\forall\,q\in H^1(I),\notag\\
  &\int_{[-1,1]}(p-\psi)d\nu=0,\label{eq:KKT2}\\[4pt]
  &d\nu=\rho\,dx+\gamma d\delta_{-1}+\zeta d\delta_1,\label{eq:KKT3}
\end{align}
  where
\begin{align}
 &\text{$\rho=p''+f'-\Phi+\lambda\in L_2(I)$ is nonnegative a.e.},\label{eq:KKT4}\\
  &\text{$\gamma=p'(-1)+f(-1)$ and $\zeta=-[p'(1)+f(1)]$
  are nonnegative numbers},\label{eq:KKT5}
\end{align}
 and $\Phi\in H^1(I)$ satisfies \eqref{eq:PhiDef1}--\eqref{eq:PhiDef2}.
\end{theorem}
\begin{remark}\label{rem:KKT}\rm
  The KKT conditions \eqref{eq:KKT1}--\eqref{eq:KKT5} are also
  sufficient conditions for \eqref{eq:VI3}.
  Indeed, they imply, for any $q\in\cK$,
\begin{align*}
  &\int_I p'(q'-p')dx+\int_I(\Phi-f')q\,dx+f(1)\big(q(1)-p(1)\big)
  -f(-1)\big(q(-1)-p(-1)\big)\\
    &\hspace{50pt}= \lambda\int_I(q-p)dx-\int_I (q-p)d\nu\\
    &\hspace{50pt}=-\int_I (q-\psi)d\nu\geq 0,
\end{align*}
 which then also implies $\bar y(x)=\int_{-1}^x p(t)dt$ is the solution of
 \eqref{eq:ROCP}.
\end{remark}
\par
 Finally we observe that Theorem~\ref{thm:KKT} implies
\begin{equation}\label{eq:KKT}
  \beta\int_I (\bar y''+f)z''dx+\int_I(\bar y-y_d)z\,dx+\int_{[-1,1]}z'd\mu=0
  \qquad\forall\, z\in H^2(I)\cap H^1_0(I),
\end{equation}
 where
\begin{equation}\label{eq:muDef}
\text{$\mu=\beta\nu$ is a nonnegative Borel measure},
\end{equation}
 and
\begin{equation}\label{eq:Complementarity}
  \int_{[-1,1]} (\bar y'-\psi)d\mu=0.
\end{equation}
%
\section{The Discrete Problem}\label{sec:Discrete}
 Let $V_h\subset H^2(I)\cap H^1_0(I)$ be the cubic
  Hermite finite element space
 (cf. \cite{Ciarlet:1978:FEM,BScott:2008:FEM}) associated with
 a triangulation/partition $\cT_h$ of $I$ with mesh size $h$.  The
 discrete problem is to find
\begin{equation}\label{eq:DiscreteProblem}
  \bar y_h=\argmin_{y_h\in K_h}\frac12\big(\|y_h-y_d\|_\LT^2
  +\beta\|y_h''+f\|_\LT^2\big),
\end{equation}
 where
\begin{equation}\label{eq:KhDef}
  K_h=\{y_h\in V_h:\, P_hy_h'\leq P_h\psi\},
\end{equation}
 and $P_h$ is the interpolation operator associated with the
 $P_1$ finite element space associated with $\cT_h$, i.e., the constraint
 \eqref{eq:PDEConstraint} is only enforced at the grid points.
\par
 The nodal interpolation operator from $C^1([-1,1])$
  onto $V_h$ is denoted by $\Pi_h$.
\par
  We will
 use the following standard estimates for $P_h$
 and $\Pi_h$ (cf. \cite{Ciarlet:1978:FEM,BScott:2008:FEM}) in the error analysis:
\begin{alignat}{3}
  \|\zeta- P_h\zeta\|_\LT&\leq Ch|\zeta|_{H^1(I)}
   &\qquad&\forall\,\zeta\in H^1(I),\label{eq:PhError1}\\
   \|\zeta-P_h\zeta\|_\LT&\leq Ch^2|\zeta|_{H^2(I)}
  &\qquad&\forall\,\zeta\in H^2(I),\label{eq:PhError2}\\
 |\zeta-\Pi_h\zeta|_{H^1(I)}+h|\zeta-\Pi_h\zeta|_{H^2(I)}&\leq Ch^2|\zeta|_{H^3(I)}
 &\qquad&\forall\,\zeta\in H^3(I).\label{eq:PihEst}
\end{alignat}
 Here and below we use $C$ to denote a generic positive constant that
  is independent of the mesh size $h$.
\par
 The unique solution $\bar y_h\in K_h$ of the minimization problem
  defined by \eqref{eq:DiscreteProblem}
 and \eqref{eq:KhDef} is characterized by the discrete variational inequality
\begin{equation*}
  \beta\int_I (\bar y_h''+f)(y_h''-\bar y_h'')dx
  +\int_I (\bar y_h-y_d)(y_h-\bar y_h)dx\geq0
  \qquad\forall\,y_h\in K_h,
\end{equation*}
 which can also be written as
\begin{equation}\label{eq:DVI}
  a(\bar y_h,y_h-\bar y_h)\geq \int_I y_d(y_h-\bar y_h)dx
  -\beta\int_I f(y_h''-\bar y_h'')dx\qquad\forall\,y_h\in K_h,
\end{equation}
 where the bilinear form $a(\cdot,\cdot)$ is defined in \eqref{eq:aDef}.
\par
 The error analysis of the finite element method is based on the
 approach in \cite{BSung:2017:State}
 for state constrained optimal control problems that was extended
  to one dimensional
 problems with constraints on the derivative of the state in \cite{BSW:2020:OneD}.
\par \par
 We will use the energy norm $\|\cdot\|_a$ defined by
\begin{equation}\label{eq:EnergyNorm}
  \|v\|_a^2=a(v,v)=\|v\|_\LT^2+\beta|v|_{H^2(I)}^2.
\end{equation}
 Note that
\begin{equation}\label{eq:NormEquivalence}
  \|v\|_a\approx \|v\|_{H^2(I)} \qquad\forall\,v\in H^2(I)\cap H^1_0(I)
\end{equation}
 by a Poincar\'e-Friedrichs inequality \cite{Necas:2012:Direct}.
%
\subsection{An Abstract Error Estimate}\label{subsec:Abstract}
 In view of \eqref{eq:DVI}, \eqref{eq:EnergyNorm}
 and the Cauchy-Schwarz inequality, we have
\begin{align}\label{eq:Preliminary}
  \|\bar y-\bar y_h\|_a^2&=a(\bar y-\bar y_h,\bar y-y_h)+
      a(\bar y-\bar y_h,y_h-\bar y_h)\notag\\
      &\leq \frac12\|\bar y-\bar y_h\|_a^2+\frac12\|\bar y-y_h\|_a^2
      +a(\bar y,y_h-\bar y_h)\\
            &\hspace{60pt}
             -\int_I y_d(y_h-\bar y_h)dx+
     \beta\int_I f(y_h''-\bar y_h'')dx\qquad\forall\,y_h\in K_h.\notag
\end{align}
\par
 It follows from  \eqref{eq:KKT}, \eqref{eq:Complementarity}  and
 \eqref{eq:KhDef} that
\begin{align}\label{eq:KKTRelation}
  &a(\bar y,y_h-\bar y_h)-\int_I y_d(y_h-\bar y_h)dx+
     \beta\int_I f(y_h''-\bar y_h'')dx\notag\\
     &\hspace{40pt}=\int_{[-1,1]} (\bar y_h'-y_h')d\mu\notag\\
     &\hspace{40pt}=\int_{[-1,1]} (\bar y_h'-P_h\bar y_h')d\mu+
\int_{[-1,1]} (P_h\bar y_h'-P_h\psi)d\mu+\int_{[-1,1]} (P_h\psi-\psi)d\mu\\
        &\hspace{90pt}+\int_{[-1,1]} (\psi-\bar y')d\mu+
         \int_{[-1,1]} (\bar y'-y_h')d\mu\notag\\
&\hspace{40pt}\leq \int_{[-1,1]} (\bar y_h'-P_h\bar y_h')d\mu
  +\int_{[-1,1]} (P_h\psi-\psi)d\mu
       +\int_{[-1,1]} (\bar y'-y_h')d\mu
       \notag
\end{align}
 for all $y_h\in K_h$.
\par
 Putting \eqref{eq:Preliminary} and \eqref{eq:KKTRelation} together, we arrive at the
 abstract error estimate
\begin{align}\label{eq:Abstract}
 \|\bar y-\bar y_h\|_a^2&\leq 2\Big(\int_{[-1,1]}
  (\bar y_h'-P_h\bar y_h')d\mu
   +\int_{[-1,1]}(P_h\psi-\psi)d\mu\Big)\\
   &\hspace{50pt} +\inf_{y_h\in K_h}\Big(\|\bar y-y_h\|_a^2+
   2\int_{[-1,1]} (\bar y'-y_h')d\mu\Big).\notag
\end{align}
%
\subsection{Concrete Error Estimates}\label{subsec:Concrete}
 The three terms on the right-hand side of \eqref{eq:Abstract}
  can be estimated as follows.
 \par
  First of all, we have
\begin{align}\label{eq:Est1}
  \int_{[-1,1]}(\bar y_h'-P_h\bar y_h')d\mu&
   =\int_{[-1,1]} \big[(\bar y_h'-\bar y')-P_h(\bar y_h'-\bar y')\big]d\mu+
     \int_{[-1,1]} (\bar y'-P_h\bar y')d\mu\notag\\
 &=\beta \Big(\int_I \big[(\bar y_h'-\bar y')-P_h(\bar y_h'-\bar y')\big]\rho\,dx+
         \int_I (\bar y'-P_h\bar y')\rho\,dx\Big)\\
      &\leq C\big(h\|\bar y-\bar y_h\|_a+h^2|y|_{H^3(I)}\big),\notag
\end{align}
 by Corollary~\ref{cor:Regularity}, \eqref{eq:KKT3}, \eqref{eq:muDef},
 \eqref{eq:PhError1},
 \eqref{eq:PhError2}, \eqref{eq:NormEquivalence}
  and the fact that $\zeta-P_h\zeta$ vanishes at the
 points $\pm1$ for any $\zeta\in H^1(I)$.
\par
 Similarly we can derive
\begin{equation}\label{eq:Est2}
  \int_{[-1,1]}(P_h\psi-\psi)d\mu=\beta\int_I (P_h\psi-\psi)\rho\,dx
  \leq Ch^2|\psi|_{H^2(I)}
\end{equation}
 by \eqref{eq:DataRegularity} and \eqref{eq:PhError2}.
\par
 Finally we have
\begin{align}\label{eq:Est3}
  &\inf_{y_h\in K_h}\Big(\|\bar y-y_h\|_a^2+
   2\int_{[-1,1]} (\bar y'-y_h')d\mu\Big)\notag\\
   &\hspace{50pt}\leq \|\bar y-\Pi_h \bar y\|_a^2+2\int_{[-1,1]}
   \big[\bar y'-(\Pi_h\bar y)'\big]d\mu\\
   &\hspace{50pt}= \|\bar y-\Pi_h \bar y\|_a^2+2\beta\int_I
   \big[\bar y'-(\Pi_h\bar y)'\big]\rho\,dx
   \leq Ch^2\big[|\bar y|_{H^3(I)}^2+|\bar y|_{H^3(I)}\big],\notag
\end{align}
 by Corollary~\ref{cor:Regularity}, \eqref{eq:KKT3}, \eqref{eq:muDef},
 \eqref{eq:PihEst},
 \eqref{eq:NormEquivalence} and the fact that
 $\bar y'-(\Pi_h\bar y)'$ vanishes at $\pm1$.
\par
 It follows from \eqref{eq:Abstract}--\eqref{eq:Est3} and Young's
 inequality that
\begin{equation}\label{eq:EnergyError}
  \|\bar y-\bar y_h\|_a\leq Ch,
\end{equation}
 which immediately implies the following result, where
 $\bar u_h=-(\bar y_h''+f)$ is the
 approximation for $\bar u=-(\bar y+f)$.
\begin{theorem}\label{thm:ErrorEstimates}
  Under the assumptions on the data in \eqref{eq:DataRegularity}, we have
\begin{equation}
  |\bar y-\bar y_h|_\HO+\|\bar u-\bar u_h\|_\LT\leq Ch.
\end{equation}
\end{theorem}
\begin{remark}\label{rem:Sharp}\rm
 Numerical results in Section~\ref{sec:Numerics} indicate that the
  estimate for $\|\bar u-\bar u_h\|_\LT$
 in Theorem~\ref{thm:ErrorEstimates} is sharp.
\end{remark}
\begin{remark}\label{rem:Comparison}\rm
  For comparison, the error estimate
\begin{equation*}
  |\bar y-\bar y_h|_\HO+\|\bar u-\bar u_h\|_\LT\leq C_\epsilon h^{\frac12-\epsilon}
\end{equation*}
 was obtained in \cite{BSW:2020:OneD} under the assumptions in \eqref{eq:OldData}.
\end{remark}
\section{A Numerical Experiment}\label{sec:Numerics}
 We begin by constructing an example for the problem
 \eqref{eq:ROCP}/\eqref{eq:VI} with a known exact solution.
\subsection{An Example}\label{subsec:Example}
 Let $\beta=1$,
\begin{equation}\label{eq:psiExample}
  \psi(x)=\begin{cases}
   1-\frac92 x^2&\qquad -1\leq x\leq 0\\[2pt]
   1&\qquad \phantom{-}0\leq x\leq 1
   \end{cases},
\end{equation}
 and
\begin{equation}\label{eq:ExactSolution}
  \bar y(x)=\int_{-1}^x p(t)dt,
\end{equation}
 where
\begin{equation}\label{eq:pDef}
  p(x)=\begin{cases}
    1-\frac{81}{32}(x-\frac13)^2&\qquad -1\leq x\leq \frac13 \\[2pt]
    1&\qquad\hspace{10pt} \frac13\leq x\leq 1
  \end{cases}.
\end{equation}
\par
 We have $\psi\in H^2(I)$,
\begin{equation}\label{eq:psiExampleConstraint}
  \int_I\psi\, dx =\frac12,
\end{equation}
 $p\in H^2(I)$,
\begin{equation}\label{eq:pSD}
  p''(x)=\begin{cases}
    -81/16& \qquad -1<x<\frac13\\[2pt]
    0&\qquad\hspace{10pt} \frac13<x<1
  \end{cases},
\end{equation}
 $p'(1)=0$, $p'(-1)=27/4$,
\begin{equation}\label{eq:pProperties}
 \int_I p\,dx=0, \quad p\leq\psi \quad\text{and}\quad \fA=\{-1\}\cup[1/3,1].
\end{equation}
\par
 Let $f\in H^1(I)$ be defined by
\begin{equation}\label{eq:fDef}
  f(x)=\begin{cases}\d
    \frac{2}{9\pi}\sin (\pi(3x-1))&\qquad -1< x\leq \frac13 \\[8pt]
    -(x-\frac13)^2 &\qquad \hspace{10pt} \frac13\leq x<1
  \end{cases}.
\end{equation}
 We have $f(-1)=0$, $f(1/3)=0$, $f_-'(1/3)=2/3$, $f'_+(1/3)=0$
  and $f(1)=-4/9$.
 Therefore the function
\begin{equation}\label{eq:PhiExample}
  \Phi(x)=\begin{cases}
    f'(x)&\qquad -1<x< \frac13\\[2pt]
    f'(x)+\frac23&\qquad\hspace{10pt} \frac13< x<1
  \end{cases}
\end{equation}
 belongs to $H^1(I)$ and
\begin{equation}\label{eq:PhiIntegralExample}
  \int_I\Phi\,dx=\int_I f'(x)+\int_\frac13^1\frac23\,dx
  =f(1)-f(-1)+\frac49=0.
\end{equation}
\par
 Finally we take $\lambda=81/16$ and $y_d=\bar y+\Phi'$.
 Then the KKT conditions \eqref{eq:KKT1}--\eqref{eq:KKT5} are
  satisfied with
\begin{align*}
  d\nu&=[p''+f'-\Phi+\lambda]dx+[p'(-1)+f(-1)]d\delta_{-1}
  -[p'(1)+f(1)]d\delta_1\\
    &=(211/48)\chi_{[1/3,1]} dx+(27/4)d\delta_{-1}+(4/9)d\delta_1,
\end{align*}
 where $\chi_{[1/3,1]}$ is the characteristic function of the
 interval $[1/3,1]$.
%
\subsection{Numerical Results}\label{subsec:Results}
 We solved the problem in Section~\ref{subsec:Example} by the finite
 element method in
 Section~\ref{sec:Discrete} on uniform meshes.  The results are
 displayed in Table~\ref{table:Results}.
\begin{table}[ht]
\begin{tabular}{|l|c|c|c|c|}
\hline
\rule{0pt}{2.5ex}
 $2/h$ &$\|\bar y-\bar y_h\|_{L_2(I)}$&$\|\bar y-\bar y_h\|_{L_\infty(I)}$
 &$|\bar y-\bar y_h|_{H^1(I)}$
  &$|\bar y-\bar y_h|_{H^2(I)}$\\[2pt]
\hline
&&&&\\[-12pt]
 $1+2^0$ & 1.430334\hspace{1pt}e-01&	1.625937\hspace{1pt}e-01&	2.581989\hspace{1pt}e-01&	8.660252\hspace{1pt}e-01\\
 $1+2^1$ & 1.216070\hspace{1pt}e-01&	1.385037\hspace{1pt}e-01&	2.199480\hspace{1pt}e-01&	7.486796\hspace{1pt}e-01\\
 $1+2^2$ & 4.306657\hspace{1pt}e-02 &	4.679253\hspace{1pt}e-02 &	8.061916\hspace{1pt}e-02 &	4.485156\hspace{1pt}e-01\\
 $1+2^3$ & 1.613494\hspace{1pt}e-02 &	1.850729\hspace{1pt}e-02	& 2.919318\hspace{1pt}e-02 &	2.573315\hspace{1pt}e-01 \\
 $1+2^4$ & 3.439341\hspace{1pt}e-03 &	3.849954\hspace{1pt}e-03	& 6.315816\hspace{1pt}e-03 &	1.266029\hspace{1pt}e-01\\
 $1+2^5$ & 9.590453\hspace{1pt}e-04 &	1.087740\hspace{1pt}e-03 &	1.741244\hspace{1pt}e-03 &	6.470514\hspace{1pt}e-02\\
 $1+2^6$ & 2.256478\hspace{1pt}e-04& 2.542346\hspace{1pt}e-04	&4.125212\hspace{1pt}e-04	& 3.223430\hspace{1pt}e-02\\
 $1+2^7$ & 5.874304\hspace{1pt}e-05	& 6.639870\hspace{1pt}e-05& 1.067193\hspace{1pt}e-04&1.618687\hspace{1pt}e-02 \\
 $1+2^8$ & 1.425640\hspace{1pt}e-05	&1.608790\hspace{1pt}e-05	 &2.549283\hspace{1pt}e-05 &	8.086258\hspace{1pt}e-03\\
 $1+2^9$ & 3.657433\hspace{1pt}e-06 &	4.124680\hspace{1pt}e-06	& 6.430499\hspace{1pt}e-06 &	4.047165\hspace{1pt}e-03\\
\hline
\end{tabular}
\bigskip
\caption{Numerical results for the example in Section~\ref{subsec:Example}}
\label{table:Results}
\end{table}
\par
 We observe $O(h)$ convergence in the $H^2$ norm which agrees
 with Theorem~\ref{thm:ErrorEstimates}.
 On the other hand the convergence in the $H^1$ norm is $O(h^2)$,
 better than the $O(h)$  convergence
predicted by Theorem~\ref{thm:ErrorEstimates}.  The convergence
in $L_2$ and $L_\infty$ is also $O(h^2)$.

\section{Concluding Remarks}\label{sec:Conclusions}
 We have shown that higher regularity for the solutions of one
 dimensional Dirichlet
  elliptic distributed optimal control problems
 with pointwise constraints on the derivative of the state can be
 obtained through a variational inequality
 satisfied by the derivative of the optimal state.  A similar result
 for one dimensional
 optimal control problems with mixed boundary conditions was obtained
  earlier in \cite{BSW:2020:OneD}.
 A natural question is: Can these results be extended to higher dimensions?
\par
 For analogs of \eqref{eq:JDef}--\eqref{eq:DerivativeConstraint} on
  a smooth/convex domain
 $\O\in \R^d$ ($d=2,3$), where $f\in H^1(\O)$ and
 $\bm{\Psi}\in [H^2(\O)]^d$,
 one can also derive a variational inequality for the gradient of
  the optimal state.
  Observe that the space $\bG$ of the gradients of the states is
  characterized by
 (cf. \cite[Chapter~I, Section~2.3]{GR:1986:NS})
\begin{align*}
 \bG&=\{\nabla y:\,y\in H^2(\O)\cap H^1_0(\O)\}\notag\\
       &=\{\bq\in[ H^1(\O)]^d:\,\Curl\bq=0 \;\text{on}\;\O\;\text{and}
       \;\bn\times\bq=0 \;
       \text{on}\;\p\O\},
\end{align*}
 where $\bn$ is the unit outward normal along $\p\O$.
\par
 Let $\bK$ be the subset of $\bG$ defined by
\begin{equation*}
 \bK=\{\bq\in\bG:\,\bq\leq\bm{\Psi}\;\text{a.e. in $\O$}\}.
\end{equation*}
 We assume that $\bK$ is nonempty, which is the case for example if
 $\bm{\Psi}\geq{\bf 0}$.
\par
  The analog of \eqref{eq:NeumannVI} is given by
\begin{equation}\label{eq:DirichletHigher}
  \int_\O  (\bPhi-\nabla f)\cdot(\bq-\bp)dx
  +\int_\O \Div\bp\,\Div(\bq-\bp)\,dx
  +\int_{\p\O} f(\bq-\bp)\cdot\bn\,dS\geq 0
\end{equation}
 for all $\bq\in \bK$,    where
 $\bp=\nabla\bar y\in\bK$,
 and $\bPhi\in \bG$ is defined by   $ \beta\Div\bPhi=y_d-\bar y$,
  which is an analog of \eqref{eq:PhiDef1}.
 The variational inequality \eqref{eq:DirichletHigher} is
  uniquely solvable because
 (cf. \cite[Chapter~I, Sections~3.2 and 3.4]{GR:1986:NS})
\begin{equation*}
  \int_\O (\Div \bq)^2dx\geq C_\O|\bq|_{H^1(\O)}^2\qquad \forall\,\bq\in \bG.
\end{equation*}
\par
 We can also write \eqref{eq:DirichletHigher} as
\begin{align}\label{eq:DirichletVector}
  &\int_\O  (\bPhi-\nabla f)\cdot(\bq-\bp)dx+
  \int_\O \big[\Div\bp\,\Div(\bq-\bp)
  +\Curl\bp\cdot\Curl(\bq-\bp)\big]\,dx\\
  &\hspace{100pt}+\int_{\p\O} f(\bq-\bp)\cdot\bn\,dS
  \geq 0\qquad\forall\,\bq\in \bK,\notag
\end{align}
 which can be interpreted as an obstacle problem for the
 vector Laplacian operator
 with natural boundary conditions.
\par
 In order  to obtain
 higher regularity for the optimal state $\bar y$, one will
 need regularity results for
\eqref{eq:DirichletHigher}/\eqref{eq:DirichletVector}, which
unfortunately are not available.
 Therefore the problem of extending the results in this paper
 to higher dimensions remains open.

%
\end{document}